\documentclass[a4paper,12pt,reqno]{amsart}
\usepackage{amssymb}
\usepackage{amsmath, amssymb, amsthm, xcolor}
\usepackage{bbm}
\usepackage{pifont}
\usepackage{tikz, tikz-cd}
\usepackage[all]{xy}
\usepackage{framed}
\usepackage{comment}
\usepackage{setspace}
\usepackage{textgreek}
\usepackage{aligned-overset}
\usepackage[shortlabels]{enumitem}
\usepackage[colorlinks=true, citecolor=blue, linkcolor=blue, bookmarks=true]{hyperref}
\tikzcdset{scale cd/.style={every label/.append style={scale=#1},
    cells={nodes={scale=#1}}}}
\newcounter{counter}

\newtheorem{theorem}[counter]{Theorem}
\newtheorem{lemma}[counter]{Lemma}

\newtheorem{definition/proposition}[counter]{Definition/Proposition}
\newtheorem{corollary}[counter]{Corollary}
\newtheorem{proposition}[counter]{Proposition}
\newtheorem*{theorem*}{Theorem}
\numberwithin{equation}{section}

\theoremstyle{definition}
\newtheorem{definition}[counter]{Definition}

\newtheorem{remark}[counter]{Remark}

\newcommand\omicron{o}

\newcommand{\ob}{\operatorname{ob}}

\newcommand{\B}{\mathcal{B}}
\newcommand{\Cu}{\operatorname{Cu}}
\newcommand{\D}{\mathcal{D}}
\newcommand{\Ro}{\mathfrak{R}}
\newcommand{\Z}{\mathcal{Z}}

\newcommand{\R}{\mathcal{R}}
\newcommand{\W}{\mathcal{W}}
\newcommand{\K}{\mathbb{K}}

\newcommand{\T}{\mathbb{T}}
\newcommand{\N}{\mathbb{N}}
\newcommand{\M}{\mathbb{M}}
\newcommand{\z}{\mathbb{Z}}

\newcommand{\Hilb}{\mathrm{Hilb}}

\newcommand{\Ad}{\mathrm{Ad}}
\newcommand{\Aut}{\mathrm{Aut}}
\newcommand{\Inn}{\mathrm{Inn}}
\newcommand{\Out}{\mathrm{Out}}

\newcommand{\id}{\operatorname{id}}

\newcommand{\Hom}{\operatorname{Hom}}

\newcommand{\au}{\approx_{a.u}}

\title[A classification of anomalous actions]{A classification of anomalous actions through model action absorption}
\author{Sergio Girón Pacheco}
\address{\hskip-\parindent Sergio Girón Pacheco, Department of mathematics, KU Leuven, Celestijnenlaan 200B, 3001, Leuven, Belgium.}
\email{sergio.gironpacheco@kuleuven.be}
\thanks{{\footnotesize  The author was supported by the Ioan and Rosemary James Scholarship awarded by St John's College and the Mathematical Institute, University of Oxford, as well as by project G085020N funded by the Research Foundation Flanders (FWO)}}

\begin{document}

\begin{abstract}
We discuss a strategy for classifying anomalous actions through model action absorption. We use this to upgrade existing classification results for Rokhlin actions of finite groups on C$^*$-algebras, with further assuming a UHF-absorption condition, to a classification of anomalous actions on these C$^*$-algebras.
\end{abstract}
\maketitle
\numberwithin{counter}{section}
\section*{Introduction}
\renewcommand*{\thecounter}{\Alph{counter}}
Connes' classification of automorphisms on the hyperfinite II$_1$ factor $\R$ (\cite{CO77,CO75}) paved the way towards a classification of symmetries of simple operator algebras. Over the next decade, this was followed by Vaughan Jones' classification of finite group actions on $\R$ (\cite{JON80}) and Ocneanu's classification of actions of countable amenable groups on $\R$ (\cite{OC06}). To achieve these classification results, an important role is played by adaptations of Connes' non-commutative Rokhlin lemma, which yields that outer group actions on $\R$ satisfy a condition often called the Rokhlin property that is analogous to properties of ergodic measure preserving actions of amenable groups on probability spaces (\cite{RO48},\cite{OW80}). In the C$^*$-setting, the analogous property is not automatic. However, there has been substantial progress in the classification of those group actions on C$^*$-algebras that satisfy the Rokhlin property (\cite{HEJO82,HEJO83,HEOC84,EVKI97,IZU04,IZU04II,NA16,GASA16}). Very recently, groundbreaking results towards a classification of group actions without the need for the Rokhlin property have appeared (\cite{GASZ22,IZMA20,IZMA21}).
\par Connes, Jones and Ocneanu also classify group homomorphisms $G\rightarrow \Out(\R)$ up to outer conjugacy (\cite{CO77,JON80,OC06}). Such a homomorphism is called a \emph{$G$-kernel} on $\R$. The classification of $G$-kernels on injective factors was completed by Katayama and Takesaki (\cite{KATA04}). These can be understood as the first classification results for quantum symmetries of $\R$ which do not arise as group actions. \emph{Quantum symmetry} is a broad term that encapsulates generalised notions of symmetry that appear in topological and conformal field theories. These symmetries are often encoded through the action of a higher category equipped with a product operation such that the category weakly resembles a group. In the case of $G$-kernels, these can be understood as actions of $2$-groups or tensor categories (\cite{JON20,EVGI21}). The study of quantum symmetries of $\R$ was developed through the subfactor theory of Jones (\cite{JON90}) culminating in Popa's classification of subfactors $\mathcal{N}\subset\R$ with amenable standard invariant (\cite{PO94}).
\par In comparison to the success in understanding the existence and classification of $G$-kernels on von Neumann algebras, the study of $G$-kernels on C$^*$-algebras has up to recently been underdeveloped. In \cite{JON20} Corey Jones studies the closely related notion of \emph{$\omega$-anomalous action}.\footnote{In the case that a C$^*$-algebra $A$ has trivial centre, the study of $\omega$-anomalous actions on $A$ is equivalent to the study of $G$-kernels on $A$ (\cite[Section 2.3]{JON20}).} In his paper, Corey Jones provides a C$^*$-adaptation of Vaughan Jones' work (\cite{JON79}), laying out a systematic way to construct anomalous actions on C$^*$-crossed products. Corey Jones also establishes existence and no-go theorems for anomalous actions on abelian C$^*$-algebras. In \cite{EVGI21} Evington and the author lay out an algebraic $K$-theory obstruction to the existence of anomalous actions on tracial C$^*$-algebras. Recently, Izumi has developed a cohomological invariant for $G$-kernels (\cite{IZU23}). This invariant introduces new obstructions to the existence of $G$-kernels which also apply in the non-tracial setting. Further, Izumi uses this invariant to classify $G$-kernels of some poly-$\z$ groups on strongly self-absorbing UCT Kirchberg algebras. 
\par This paper provides a classification of anomalous actions with the Rokhlin property on C$^*$-algebras where $K$-theoretic obstructions vanish. The Rokhlin property for finite group actions was first systematically studied by Izumi (\cite{IZU04,IZU04II}). In his work, Izumi uses the Rokhlin property to boost existing classification results of Kirchberg algebras in the UCT class (\cite{PHI00,Ki95}) and unital, simple, separable, nuclear, tracially approximate finite dimensional (TAF) algebras in the UCT class (\cite{LIN04}) by their K-theory, to a classification of finite group actions with the Rokhlin property on these classes of C$^*$-algebras by the induced module structure on $K$-theory (\cite[Theorem 4.2,\ Theorem 4.3]{IZU04II}).\footnote{TAF algebras are C$^*$-algebras that may be locally approximated by finite dimensional C$^*$-algebras in trace (see \cite[Definition 1,2]{LIN01preclass}).}
\par The strategy of this paper is to bootstrap Izumi's classification of $G$ actions with the Rokhlin property, for finite groups $G$, to achieve analogous classification results for anomalous actions. To do this, we will assume that our C$^*$-algebra $A$ satisfies a UHF absorbing condition. To be precise, that the $A$ is stable under tensoring with the UHF algebra $M_{|G|^\infty}\cong \bigotimes_{i\in\N} M_{|G|}$. This property is considered for example in \cite{SZBA17} and \cite{GASA16} and in some cases follows immediately from the existence of Rokhlin $G$ actions on $A$ (\cite[Theorems 3.4 and 3.5]{IZU04II},\cite[Theorem 5.2]{GASA16}). Further assuming the Rokhlin property, we will establish a model action absorption result (Proposition \ref{prop:absorption}). Second, we will use the model action absorption combined with a trick, that builds on ideas of Connes in the cyclic group case (\cite[Section 6]{CO77}). This trick lets us use the existence of anomalous action on the UHF-algebra $M_{|G|^\infty}$ to reduce the classification of anomalous actions to the classification of cocycle actions. We may not apply this method by replacing $M_{|G|^\infty}$ by $\Z$ or $\mathcal{O}_\infty$ due to the obstruction results of \cite[Theorem A]{EVGI21} and \cite[Theorem 3.6]{IZU23}. This argument allows us to prove the following.
\begin{theorem}\label{thm:classificationKirchberg+TAFintro}{(cf. Theorem \ref{thm:Rokhclasskirchberg} and Theorem \ref{thm:RokhlinclassTAF})}
    Let $G$ be a finite group and $A\cong A\otimes M_{|G|^\infty}$ be either a Kirchberg algebra in the UCT class or a unital, simple, separable, nuclear TAF algebra in the UCT class. If $(\alpha,u),(\beta,v)$ are anomalous $G$ actions on $A$ with the Rokhlin property, then $(\alpha,u)$ is cocycle conjugate to $(\beta,v)$ through an automorphism that is trivial on $K$-theory if and only if $K_i(\alpha_g)=K_i(\beta_g)$ for all $g\in G$ and the anomalies of $(\alpha,u)$ and $(\beta,v)$ coincide.
\end{theorem}
\par  Similarly, we can boost Nawata's classification of Rokhlin $G$ actions on $\mathcal{W}$ (see \cite{NA16}) to a classification of anomalous actions on $\W$.
\begin{theorem}{(cf. Theorem \ref{thm:classW})}
    Let $G$ be a finite group and $(\alpha,u), (\beta,v)$ be anomalous $G$ actions on $\W$ with the Rokhlin property, then $(\alpha,u)$ is cocycle conjugate to $(\beta,v)$ if and only if the anomalies of $(\alpha,u)$ and $(\beta,v)$ coincide.
\end{theorem}
\par As a consequence of the results of \cite{GASA16}, we may also apply this strategy to classify anomalous actions with the Rokhlin property on C$^*$-algebras that arise as inductive limits of 1-dimensional NCCW complexes (see Theorem \ref{thm:1dimNCCW}).

\par The procedure utilised for the proof of Theorem \ref{thm:classificationKirchberg+TAFintro} can be expected to work in more generality.  The reason we restrict to unital, simple, nuclear TAF algebras in the tracial setting is due to the need to apply classification results for (cocycle) group actions. With more novel stably finite classification results in hand (\cite{Class1}), and using similar techniques to \cite{IZU04,IZU04II}, a classification of finite group actions with the Rokhlin property on simple, separable, nuclear, $\Z$-stable C$^*$-algebra satisfying the UCT through the induced module structure on the Elliott invariant is plausible. A strategy to approach this classification problem has been proposed by Szab\'o in private communications. With such a result in hand, one could apply the abstract Lemma \ref{lem:abstract} to yield the equivalent to Theorem \ref{thm:classificationKirchberg+TAFintro} in the generality of simple, separable, nuclear, $M_{|G|^\infty}$-stable C$^*$-algebras satisfying the UCT.
\par Recent advances in the classification of more general symmetries on C$^*$-algebras pave the way towards a classification of quantum symmetries. Significant results in this direction are the classification of AF-actions of fusion categories on AF-algebras (\cite{QCCJRHP22}), as well as Yuki Arano's announcement of an adaptation of Izumi's techniques in \cite{IZU04} to actions of fusion categories with the Rokhlin property. In the final section of this paper, we connect our results to the work in \cite{QCCJRHP22}. We demonstrate the existence of an AF $\omega$-anomalous $G$-action with the Rokhlin property on $M_{|G|^\infty}$ which we denote by $\theta_G^{\omega}$. This has structural implications for anomalous actions with the Rokhlin property on any AF-algebra $A$. Indeed, combined with Theorem \ref{thm:classificationKirchberg+TAFintro}, the existence of $\theta_G^{\omega}$ implies that every anomalous action on $A$ with the Rokhlin property, that consists of automorphisms that act trivially on $K$-theory, is automatically AF (see Corollary \ref{cor:AF}).  Under some assumptions on the anomaly, an application of the classification results of \cite{QCCJRHP22} establish the converse (see Corollary \ref{cor:AF}). This partial converse exhibits a difference in behavior between anomalous actions and group actions (see the discussion following Corollary \ref{cor:AF}).
\par The paper is organised as follows. In Section \ref{sec:prelim} we recall some necessary background on anomalous actions. Section \ref{sec:UHFactions} recalls the construction of model anomalous actions on UHF algebras. In Section \ref{sec:absorption} we prove a model action absorbing result for finite group anomalous actions. In Section \ref{sec:classification} we set out an abstract lemma for the classification of anomalous actions (Lemma \ref{lem:abstract}) which we use to prove our main results. Finally, in Section \ref{sec:applications}, we discuss an application of the classification result to AF-actions.
\subsection*{Acknowledgements} The author would like to thank Samuel Evington, Eusebio Gardella, Andr\'e Henriques, Corey Jones, Ulrich Pennig and Stuart White for comments and discussions that have been useful for this paper. An initial version of this work forms part of the authors DPhil thesis \cite{THESIS}.
\section{Preliminaries}\label{sec:prelim}
\numberwithin{counter}{section}
Throughout, $A$ and $B$ will be used to denote C$^*$-algebras and $G,\Gamma,K$ will be used to denote countable discrete groups. We let $\T\subset \mathbb{C}$ be the circle group. We denote the multiplier algebra of $A$ by $M(A)$. Any automorphism $\alpha\in \Aut(A)$ extends uniquely to an automorphism of $M(A)$, we denote this extension also by $\alpha$. For a unitary $u\in M(A)$ we write $\Ad(u)$ for the automorphism $a\mapsto uau^*$ of $A$ and the group of inner automorphisms on $A$ by $\Inn(A)$. Recall that a $G$-kernel of $A$ is a group homomorphism $G\rightarrow \Aut(A)/\Inn(A)=\Out(A)$. We now recall the definition of an anomalous action from \cite[Definition 1.1]{JON20}. In the case that $A$ has trivial centre this notion coincides with a lift of a $G$-kernels into $\Aut(A)$. 
\begin{definition}\label{def:anomaction}
An \emph{anomalous action} of a countable discrete group $G$ on a C$^*$-algebra $A$ consists of a pair $(\alpha,u)$ where
$$\alpha: G\rightarrow \Aut(A)$$
$$u:G\times G\rightarrow U(M(A))$$ 
are a pair of maps such that
\begin{align}
    \alpha_g\alpha_h&=\Ad(u_{g,h})\alpha_{gh},\ \text{for all}\ g,h\in G,\label{eqn:omega1}\\ 
    \alpha_g(u_{h,k})&u_{g,hk}u_{gh,k}^*u_{g,h}^* \in \T\cdot 1_{M(A)},\ \text{for all}\ g,h,k\in G. \label{eqn:omega}
\end{align}
\end{definition}
Firstly, note that in (\ref{eqn:omega1}) and (\ref{eqn:omega}) we have used the subscript notation $\alpha_g$ and $u_{g,h}$ instead of $\alpha(g)$ and $u(g,h)$ for $g,h\in G$. We will use this throughout when notationally convenient.
\par As shown in \cite[Lemma 7.1]{EILMAC47} the formula in (\ref{eqn:omega}) defines a circle valued $3$-cocycle i.e. an element of $Z^3(G,\T)$. We will call this the \emph{anomaly} of the action and denote it by $o(\alpha,u)$. For $\omega\in Z^3(G,\T)$ we say $(\alpha,u)$ is a $(G,\omega)$ action on $A$ to mean that $(\alpha,u)$ is an anomalous action of $G$ on $A$ with anomaly $\omega$.\footnote{In \cite{JON20} the anomaly $\omega$ is carried as part of the data. We prefer to see the anomaly as an invariant of the pair $(\alpha,u)$.} If $\omega=1$ then we call $(\alpha,u)$ a \emph{cocycle action}. Note that any anomalous action $(\alpha,u)$ induces a $G$-kernel when passing to the quotient group $\Out(A)$, we denote its associated $G$-kernel by $\overline{\alpha}$. For any $G$-kernel $\overline{\alpha}$ on $A$ we denote by $\ob(\overline{\alpha})\in H^3(G,Z(U(M(A))))$ its $3$-cohomology invariant (see e.g. \cite[Section 2.1]{EVGI21}).
\par The reader should be warned that there is a slight variation in Definition \ref{def:anomaction} to the definitions of anomalous actions in \cite{JON20} and \cite{EVGI21}. Given our conventions in Definition \ref{def:anomaction}, a $(G,\omega)$ action induces an $\overline{\omega}$ anomalous action as in \cite[Definition 1.1]{JON20}, this is seen by taking $m_{g,h}=u_{g,h}^*$.
\par Throughout this paper, we will denote the algebra of bounded sequences of $A$ quotiented by those sequences going to zero in norm by $A_\infty$. For a $^*$-closed subset $S$ of $A_{\infty}$ we may consider the commutant C$^*$-algebra $A_\infty\cap S'=\{x\in A_{\infty}:[x,S]=0\}$ and the annihilator $A_{\infty}\cap S^{\perp}=\{x\in A_{\infty}: xS=Sx=0\}.$
We may then denote Kirchberg's sequence algebra by 
$$F(S,A_{\infty})=(A_{\infty}\cap S')/(A_{\infty}\cap S^{\perp}).$$
In the case that $S$ is the C$^*$-algebra of constant sequences in $A_{\infty}$ we denote this simply by $F(A)=F(A,A_{\infty})$ and $F(A)$ the \emph{central sequence algebra} of $A$. Note that $F(A)$ is a unital C$^*$-algebra whenever $A$ is $\sigma$-unital. Indeed, the unit is given by $h=(h_n)$ for any sequential approximate unit $h_n$ for $A$.
\par Any automorphism $\theta\in\Aut(A)$ induces an automorphism $\theta$ of $A_\infty$ through $(a_n)\mapsto(\theta(a_n))$ for any $(a_n)\in A_\infty$.\footnote{Note the abuse of notation.} If a subset $S$ of $A_\infty$ is invariant under both $\theta$ and $\theta^{-1}$, then so are $A_\infty\cap S'$ and $A_\infty \cap S^{\perp}$ and $\theta$ induces an automorphism of $F(S,A_\infty)$. 
\begin{remark}\label{rem:innertrivialoncentral}
When $A$ is equipped with a $(G,\omega)$ action $(\alpha,u)$, it induces a $(G,\omega)$ action on $A_\infty$. In fact, $\alpha$ induces a group action on $F(A)$ as $\Ad(u)(x)-x\in A_\infty \cap A^{\perp}$ for any $x\in A_\infty\cap A'$ and $u\in U(M(A))$. Similarly, if $S=S^*$ is an $\alpha$ invariant subset of $A_{\infty}$ containing $A$ and $S$ is also invariant by $u_{g,h}$ for all $g,h\in G$ (i.e. $u_{g,h}S+Su_{g,h}\in S$ for all $g,h\in G$), then $\alpha$ induces a group action on $F(S,A_{\infty})$ (see \cite[Remark 1.8]{SZI}). A subset which is invariant under both $\alpha$ and $u$ will be called \emph{($\alpha,u$)-invariant}.
\end{remark}
\par We will be interested in anomalous actions with the Rokhlin property. This notion was introduced in \cite[Definition 3.10]{IZU04} for actions of finite groups on unital C$^*$-algebras and later generalised by Nawata and Santiago for non-unital C$^*$-algebras (see \cite{NA16} and \cite{SA15}). Its definition in the setting of anomalous actions is ad verbatim, we will only require it for $\sigma$-unital C$^*$-algebras.
\begin{definition}\label{def:Rokhlin}
An anomalous action $(\alpha,u)$ of a finite group $G$ on a $\sigma$-unital C$^*$-algebra $A$ is said to have the \emph{Rokhlin property}, if there exist projections $p_g\in F(A)$ for $g\in G$ such that:
\begin{enumerate}
    \item $\sum_{g\in G} p_g=1,$ 
    \item $\alpha_g(p_h)=p_{gh}.$
\end{enumerate}
\end{definition}
\begin{remark}
The Rokhlin property also makes sense for $G$-kernels. In this case, a $G$-kernel $\overline{\alpha}$ of a finite group $G$ on a $\sigma$-unital C$^*$-algebra $A$ satisfies the Rokhlin property if for any/some lift $(\alpha,u)$ of $\overline{\alpha}$ there exists a partition of unity of projections $p_g\in F(A)$ for $g\in G$ such that $\alpha_g(p_h)=p_{gh}$ for all $g,h\in G$.
\end{remark}
Our main goal is to classify anomalous actions with the Rokhlin property. To make sense of this question, we first need to introduce equivalence relations for anomalous actions. Before we do so, we start by introducing some notation that will allow us to streamline future definitions.
\begin{definition}\label{def:perturbation}
    Let $(\alpha,u)$ be an anomalous action of a group $G$ on a C$^*$-algebra $A$. If $\mathbbm{v}_g\in U(M(A))$ for $g\in G$, then the pair $(\alpha^{\mathbbm{v}},u^{\mathbbm{v}})$ with
    \begin{equation*}
        \alpha^{\mathbbm{v}}_g=\Ad(\mathbbm{v}_g)\alpha_g,\quad g\in G,
    \end{equation*}
    \begin{equation*}
        u^{\mathbbm{v}}_{g,h}=\mathbbm{v}_g\alpha_g(\mathbbm{v}_h)u_{g,h}\mathbbm{v}_{gh}^*,\quad g,h\in G
    \end{equation*}
    is an anomalous action. We say that $(\alpha^{\mathbbm{v}},u^{\mathbbm{v}})$ is a \emph{unitary perturbation} of $(\alpha,u)$.
\end{definition}
It is a straightforward that $o(\alpha,u)=o(\alpha^{\mathbbm{v}},u^{\mathbbm{v}})$ for any map $\mathbbm{v}:G\rightarrow U(M(A))$.
    \begin{definition}\label{def:equivalences}
    Let $A,B$ be C$^*$-algebras, $(\alpha,u)$ be an anomalous $G$ action on $A$ and $(\beta,v)$ be an anomalous action on $B$. Then we say that
    \begin{enumerate}[(i)]
        \item $(\alpha,u)$ is \emph{conjugate} to $(\beta,v)$ if there exists an isomorphism $\theta:A\rightarrow B$ such that $\alpha_g=\theta\beta_g\theta^{-1}$ and $v_{g,h}=\theta(u_{g,h})$ for all $g,h\in G$.
        \item $(\alpha,u)$ is \emph{cocycle conjugate} to $(\beta,v)$ if there exist unitaries $s_g\in U(M(A))$ for $g\in G$ such that $(\alpha^{s},u^{s})$ is conjugate to $(\beta,v)$. We denote this by $(\alpha,u)\simeq (\beta,v)$.\label{coconj}
        \item If $A$ and $B$ are equal and $(\alpha,u)\simeq(\beta,v)$ with the conjugacy holding through an automorphism $\theta$ such that $K_i(\theta)=\id_{K_i(A)}$ for $i=1,2$,  we say $(\alpha,u)$ and $(\beta,v)$ are \emph{$K$-trivially cocycle conjugate}. We denote this by $(\alpha,u)\simeq_{K}(\beta,v)$. \label{item:Kconj}
        \end{enumerate}
\end{definition}
\par Finally, recall the definition of a unitary one cocycle.
\begin{definition}
Let $\alpha$ be a $(G,\omega)$ action on a C$^*$-algebra $A$. We call a map $v: G\rightarrow U(M(A))$ such that $v_g\alpha_g(v_h)=v_{gh}$ an \emph{$\alpha$-cocyle}. 
\end{definition}
\section{Model actions}\label{sec:UHFactions}
Given a finite group $G$ and $\omega\in Z^3(G,\T)$ a $3$-cocycle, \cite[Theorem C]{EVGI21} constructs a $(G,\omega)$ action on $M_{|G|^\infty}$. This result is based on a construction of Corey Jones in \cite{JON20} which in turn is based on a construction of Vaughan Jones in the setting of von Neumann algebras (\cite{JON79}).
\par In this section, we recall this construction as we will need its specific form to deduce properties of the action. In \cite{JON20}, Corey Jones shows that if $\omega$ is a normalised $3$-cocycle and one has the following data:
\begin{itemize}
    \item A group $\Gamma$ and a surjection $\rho:\Gamma\twoheadrightarrow G$ such that $\rho^*(\omega)$ is a coboundary.
    \item A normalized $2$-cochain $c:\Gamma\times \Gamma\rightarrow \mathbb{T}$ such that $\rho^*(\omega)=dc$.
    \item A C$^*$-algebra $B$ and an action $\pi:\Gamma\rightarrow \Aut(B)$.
\end{itemize}
one can induce a $(G,\omega)$ action on the twisted reduced crossed product $B\rtimes_{\pi,\overline{c}}^r K$, with $K=\ker(\rho)$ (see \cite{BS70} for a reference on twisted crossed products).\footnote{For $c\in C^2(G,\T)$ we denote by $\overline{c}$ the $2$-cochain given by $\overline{c}_{g,h}=\overline{c_{g,h}}$ for $g,h\in G$.} The automorphic data of this $(G,\omega)$ action is given by
\begin{equation}\label{eqn: modelaction}
\theta_g(\sum_{k\in K}a_kv_k)=\sum_{k\in K} c_{\hat{g}k\hat{g}^{-1},\hat{g}^{-1}}\overline{c_{\hat{g},k}}\pi_{\hat{g}}(a_k)v_{\hat{g}k\hat{g}^{-1}},
\end{equation}
for $a_k\in B$, $v_k$ the canonical unitaries in $M(B\rtimes_{\pi,\overline{c}}^r K)$, $g\in G$ and $g\mapsto \hat{g}$ a choice of set theoretic section to $\rho: \Gamma\rightarrow G$.\footnote{Note that (\ref{eqn: modelaction}) is different to the formula in \cite[Lemma 3.2]{JON20}. This is due to our change of conventions when defining anomalous actions.} In fact, given an arbitrary finite group $G$, Corey Jones constructs a finite group $\Gamma$, a surjection $\rho$ and a $2$ cochain $c$ with the conditions needed above and additionally $c|_{\ker(\rho)}=1$. Additionally to $\Gamma$ and $c$, the extra data considered in \cite[Theorem C]{EVGI21} is:
\begin{itemize}
\item $B=\bigotimes_{i\in \N} \B(l^2(\Gamma))$,
\item $\pi=\Ad(\lambda_\Gamma)^{\otimes \infty}$,
\end{itemize}
with $\lambda_\Gamma$ the left regular representation and $\Ad(\lambda_\Gamma)_\gamma(T)=\lambda_\Gamma(\gamma)T\lambda_\Gamma(\gamma)^*$ for all $T\in\B(l^2(\Gamma))$ and $\gamma\in \Gamma$. In this case, the crossed product $B\rtimes_\pi^r K$ is shown to be isomorphic to the UHF algebra $M_{|G|^\infty}$. Corey Jones' construction then yields a $(G,\omega)$ action on $M_{|G|^\infty}$ through (\ref{eqn: modelaction}) for any $\omega\in Z^3(G,\T)$, we denote it by $(s_G^\omega,u_G^{\omega})$.
\begin{proposition}\label{prop:rokhlin}
Let $G$ be a finite group and $\omega\in Z^3(G,\T)$ then
$(s_G^\omega,u_G^{\omega})$ has the Rokhlin property.
\end{proposition}
\begin{proof}
We use the notation set up in the previous paragraphs. Furthermore, denote by $r_i:\B(l^2(\Gamma))\rightarrow B$ the unital embedding into the $i$-th tensor factor. As $A=B\rtimes_{\pi} K$ is unital $F(A)$ coincides with $A_{\infty}\cap A'$ so it suffices to find a partition of unity $p_g\in A_{\infty}\cap A'$ for $g\in G$ such that $\alpha_g(p_h)=p_{gh}$ for all $g,h\in G$.
\par Let $e_K$ in $\mathcal{B}(l^2(\Gamma))$ be the projection onto $l^2(K)$, that is $$e_K(\sum_{\gamma\in \Gamma}\mu_\gamma  \gamma)=\sum_{\gamma\in K}\mu_\gamma \gamma$$
for any complex scalars $\mu_\gamma$. Let $p_n=r_n(e_K)$ for $n\in \N$. Note that the projection $p=(p_n) \in B_\infty$ commutes with any constant sequence of elements in $B$. Moreover, $p$ commutes with the subalgebra $C^*(K)\subset (B\rtimes K)_\infty$. Indeed, $e_K$ is invariant under $\Ad(\lambda_\Gamma)_{k}$ for any $k\in K$ and therefore for any $n\in \N$ and $k\in K$
\begin{align*}
    v_kp_nv_k^*&=\Ad(\lambda_\Gamma)_{k}^{\otimes \infty}(r_n(e_K))\\
    &=r_n(\Ad(\lambda_\Gamma)_{k}e_K))\\
    &=r_n(e_K)\\
    &=p_n
\end{align*}
Therefore, $p\in A_{\infty}\cap A'$. 
\par We claim that the projections $p_g:=s_G^{\omega}(g)(p)=(s_G^{\omega}(g)(p_n))_{n\in \N}$ form a set of Rokhlin projections. We start by showing that the sum $\sum_{g\in G} s_G^\omega(g)(p)=1$. Let $n\in \N$ and $g\in G$, then as the cocycle $c$ is normalised, it follows from (\ref{eqn: modelaction}) that
\begin{align}
    s_G^{\omega}(g)(p_n)&=\pi_{\hat{g}}(p_n)\label{eqn:rokhlin1}\\
    &=\Ad(\lambda_{\Gamma})^{\otimes\infty}_{\hat{g}}(p_n)\notag \\
    &=\Ad(\lambda_{\Gamma})^{\otimes\infty}_{\hat{g}}(r_n(e_K))\notag \\
    &=r_n(\Ad(\lambda_{\Gamma})_{\hat{g}}(e_K))\notag.
\end{align}
The maps $r_n$ are unital so it suffices to show that $\sum_{g\in G}\Ad(\lambda_{\Gamma})_{\hat{g}}(e_K)=1_{\B(l^2(\Gamma))}$. To see this, let $\gamma\in \Gamma$, $g\in G$ and $\delta_\gamma\in l^2(\Gamma)$ the point mass at $\gamma$, then
\begin{align}
    \Ad(\lambda_{\Gamma})_{\hat{g}}(e_K)(\delta_\gamma)&=\lambda_\Gamma(\hat{g})e_K\lambda_\Gamma(\hat{g}^{-1})(\delta_\gamma)\label{eqn:rokhlin2}\\
    &=\lambda_\Gamma(\hat{g})e_K(\delta_{\hat{g}^{-1}\gamma})\notag\\
    &=\begin{cases}\notag
        \delta_{\gamma}\ \text{if}\ \gamma\in \hat{g}K,\\
        0\ \text{otherwise}.
    \end{cases}
\end{align}
The left $K$ cosets are pairwise disjoint and cover the whole group $\Gamma$. Therefore, it follows that $\sum_{g\in G}\Ad(\lambda_{\Gamma})_{\hat{g}}(e_K)(\delta_\gamma)=\delta_\gamma$  for every $\gamma\in \Gamma$. As the operators $\sum_{g\in G}\Ad(\lambda_{\Gamma})_{\hat{g}}(e_K)$ and $\id_{\B(l^2(\Gamma))}$ coincide on a spanning set of $l^2(\Gamma)$, these operators are equal.
\par It remains to show that for $g,h\in G$ the projections $s_G^{\omega}(g)p_h=p_{gh}$. This follows as $s_G^{\omega}(g)p_h=s_G^{\omega}(g)s_G^{\omega}(h)p=\Ad(u^{\omega}_G(g,h))s_G^{\omega}(gh)p=\Ad(u_G^{\omega}(g,h))p_{gh}=p_{gh}$ where the last equality in the chain holds as $p_{gh}$ commutes with $A$.
\end{proof}
\section{Absorption of model actions}\label{sec:absorption}
\par In this section we show that any Rokhlin anomalous action of a finite group $G$, on an $M_{|G|^\infty}$-stable C$^*$-algebra, absorbs the action 
$$s_G=\bigotimes_{i=0}^{\infty}\Ad(\lambda_G)$$
up to cocycle conjugacy.\footnote{Note that for $G$ finite the C$^*$-algebras $M_{|G|}$ and $B(l^2(G))$ are canonically isomorphic, we identify them throughout this paper.} This result is similar in nature to $(i)\Rightarrow (iii)$ of \cite[Theorem 5.2]{GASA16}. The methods utilised in this chapter are an adaptation of Vaughan Jones' work (\cite{JON80}) to the C$^*$-setting.
\par In his work \cite{SZI,SZII,SZIII}, Szab\'o establishes the theory of strongly self-absorbing C$^*$-dynamical systems as an equivariant version of strongly self-absorbing C$^*$-algebras that were introduced in \cite{TOWI07}. We recall the main definition below.
\begin{definition}\label{def:stronglyselfabsorb}
Let $G$ be a locally compact group. A group action $\gamma$ on a unital, separable C$^*$-algebra $\mathcal{D}$ is called \emph{strongly self-absorbing} if there exists an equivariant isomorphism $\varphi:(\mathcal{D},\gamma)\rightarrow (\mathcal{D}\otimes \mathcal{D},\gamma\otimes\gamma)$ such that there exist unitaries $u_n\in U(\mathcal{D}\otimes \mathcal{D})$ fixed by $\gamma\otimes\gamma$ with 
$$\lim\limits_{n\rightarrow \infty}\|\varphi(a)-u_n(a\otimes 1_{\mathcal{D}})u_n^*\|=0$$
for all $a\in \mathcal{D}$.
\end{definition}
The relevant example of a strongly self-absorbing action for this paper is $s_G$. That $s_G$ is strongly self-absorbing follows as a consequence of \cite[Example 5.1]{SZI}.
\par In  \cite[Theorem 3.7]{SZI} Szab\'o shows equivalent conditions for a cocycle action to tensorially absorb a strongly self-absorbing action. Although Szab\'o's theory only treats the case of cocycle actions absorbing a given strongly self-absorbing group action, many of the arguments follow in exactly the same way when replacing cocycle actions by anomalous actions that may have non-trivial anomaly. The proofs of \cite[Lemma 2.1,\ Theorem 2.6]{SZI} and \cite[Theorem 3.7, Corollary 3.8]{SZI} for example, make no use of the anomaly associated to $(\alpha,u)$ and $(\beta,w)$ being trivial. Under this observation, we can state a specific case of \cite[Corollary 3.8]{SZI}.
\begin{theorem}[{cf. \cite[Theorem 2.8]{SZI}}]\label{thm:Gabor}
 Let $A$ and $\mathcal{D}$ be separable C$^*$-algebras and $G$ a finite group. Assume $(\alpha,u):G\curvearrowright A$ is an anomalous action. Let $\gamma:G\curvearrowright \mathcal{D}$ be a group action such that $(\mathcal{D},\gamma)$ is strongly self-absorbing. If there exists an equivariant and unital $^*$-homomorphism
 $$(\mathcal{D},\gamma)\rightarrow (F(A),\alpha)$$
 then $(A,\alpha,u)$ is cocycle conjugate to $(A\otimes \mathcal{D},\alpha\otimes\gamma,u\otimes 1_{\mathcal{D}})$ through a map $\varphi:A\rightarrow A\otimes \mathcal{D}$ that is approximately unitarily equivalent to $\id_A\otimes 1_{\mathcal{D}}$.
\end{theorem}
\par We still require a few more results before we can achieve the model action absorption. These are based on known results in the setting of finite group actions on unital C$^*$-algebras. These generalise line by line to anomalous actions of finite groups on unital C$^*$-algebras, we adapt the arguments also for non-unital C$^*$-algebras.
\begin{lemma}[{cf. \cite[Theorem 3.3]{HIWI07}}]\label{lem:embeddingintofixedpoints}
Let $A$ be a C$^*$-algebra, $G$ be a finite group and $(\alpha,u)$ be an anomalous action of $G$ on $A$ with the Rokhlin property. If $B=B^*$ is a separable $(\alpha,u)$-invariant subset of $A_\infty$ and there exists a unital $^*$-homomorphism $M\rightarrow F(B,A_{\infty})$ for some separable, unital C$^*$-algebra $M$, then there exists a unital $^*$-homomorphism $M\rightarrow F(B,A_\infty)^{\alpha}$.
\end{lemma}
\begin{proof}
Fix a unital homomorphism $\psi: M\rightarrow F(B,A_{\infty})$ and choose a linear lift $\psi_0:M\rightarrow A_{\infty}\cap B'$. Then one has that
\begin{enumerate}[(i)]
\item $(\psi_0(m)\psi_0(m')-\psi_0(mm'))b=0,\quad \forall m,m'\in M, b\in B,$\label{enum:lift1}
\item $(\psi_0(m^*)-\psi_0(m)^*)b=0,\quad \forall m\in M, b\in B,$\label{enum:lift2}
\item $\psi_0(1)b-b=0,\quad \forall b\in B.$\label{enum:lift3}
\end{enumerate}
Let $S=B\cup_{g\in G}\alpha_g(\psi_0(M))\cup_{g\in G}\alpha_g(\psi_0(M))^*$ so $S=S^*$. By the Rokhlin property followed by a standard reindexing argument, there exist positive contractions $f_g\in A_{\infty}\cap S'$ such that 
\begin{enumerate}[(iv)]
\item $(\alpha_g(f_h)-f_{gh})a=0,\quad \forall g,h\in G, a\in S,$\label{enum:rokhlin1}
\item $(\sum_{g\in G} f_g)a-a=0\quad \forall a\in S,$\label{enum:rokhlin2}
\item $f_gf_ha-\delta_{g,h}a=0\quad \forall g,h\in G, a\in S.$\label{enum:rokhlin3}
\end{enumerate}
Now consider the linear mapping $\varphi:M\rightarrow A_{\infty}\cap B'$ given by
$$\varphi(m)=\sum_{g\in G}\alpha_g(\psi_0(m))f_g$$
Firstly, for $m,m'\in M$ and $b\in B$ it follows from \ref{enum:lift1} and \ref{enum:rokhlin3}
\begin{align*}
    \varphi(m)\varphi(m')b&=\sum_{g,h\in G}\alpha_g(\psi_0(m))f_g\alpha_h(\psi_0(m'))f_hb\\
    &=\sum_{g,h\in G} \alpha_g(\psi_0(m))\alpha_h(\psi_0(m'))f_gf_hb\\
    &=\sum_{g\in G}\alpha_g(\psi_0(m))\alpha_g(\psi_0(m'))bf_g\\
    &=\sum_{g\in G}\alpha_g(\psi_0(mm'))bf_g\\
    &=\sum_{g\in G}\alpha_g(\psi_0(mm'))f_gb\\
    &=\varphi(mm')b
\end{align*}
Also for $k\in G, m\in M$ and $b\in B$ it follows using \ref{enum:rokhlin1} that
\begin{align*}
\alpha_k(\varphi(m))b&=\sum_{g\in G}\alpha_k\left(\alpha_g(\psi_0(m))\right)\alpha_k(f_g)b\\
&=\sum_{g\in G}\Ad(u_{k,g})(\alpha_{kg}(\psi_0(m)))f_{kg}b\\
&=\sum_{g\in G}\Ad(u_{k,g})(\alpha_{kg}(\psi_0(m)))bf_{kg}\\
&=\varphi(m)b.
\end{align*}
Where in the last line we have used that $B$ is $u$ invariant and so the observation in Remark \ref{rem:innertrivialoncentral} applies. Therefore, the map
$$m\mapsto \varphi(m)+A_\infty\cap B^{\perp}$$
defines a homomorphism from $M$ into $(F(B,A_\infty))^\alpha$. This homomorphism is unital through combining \ref{enum:lift3} and \ref{enum:rokhlin2} and $^*$-preserving by \ref{enum:lift2}.
\end{proof}
\par In the next lemma, recall that if $\alpha$ is a action of a group $G$ on a C$^*$-algebra $A$, an $\alpha$-cocycle is a family of unitaries $v_g\in U(M(A))$ for $g\in G$ such that $v_g\alpha_g(v_h)=v_{gh}$.
\begin{lemma}[cf. {\cite[Lemma III.1]{HEJO83}}]\label{lemma: stability}
Let $A$ be a separable C$^*$-algebra and $G$ a finite group. Let $(\alpha,u)$ be an anomalous action of $G$ on $A$ with the Rokhlin property. Let $B=B^*$ be a separable $(\alpha,u)$-invariant subset of $A_\infty$. For any $\alpha$-cocycle $v_g$ for the action induced by $\alpha$ on $F(B,A_\infty)$ there exists a unitary $u\in F(B,A_\infty)$ with $u^*\alpha_g(u)=v_g$.
\end{lemma}
\begin{proof}
let $v_g\in U(F(B,A_\infty))$ be an $\alpha$-cocycle. Choosing lifts $v_g'\in A_{\infty}\cap B'$ for $v_g$ one has
\begin{enumerate}[(i)]
    \item $v_g'(v_g')^*b-b=0,\quad \forall g\in G, b\in B,$\label{enum2:lift1}
    \item $(v_g')^*v_g'b-b=0,\quad \forall g\in G, b\in B,$\label{enum2:lift2}
    \item $v_g'\alpha_g(v_h')b-v_{gh}'b=0,\quad \forall g,h\in G, b\in B.$ \label{enum2:lift3}
\end{enumerate}
Let $S=B\cup\{\alpha_h(v_g'), \alpha_h(v_g')^* :g,h\in G\}$. As in the previous lemma, one may apply the Rokhlin property combined with a reindexing argument to get a family of positive elements $f_g\in A_{\infty}\cap S'$ such that 
\begin{enumerate}[(iv)]
    \item $(\alpha_g(f_h)-f_{gh})a=0,\quad \forall g,h\in G, a\in S,$\label{enum2:rokhlin1}
    \item $\sum_{g\in G}f_ga-a=0, \quad \forall a\in S,$\label{enum2:rokhlin2}
    \item $f_gf_ha-\delta_{g,h}a=0\quad \forall g,h\in G, a\in S.$\label{enum2:rokhlin3}
\end{enumerate}

Let $u=\sum_{g\in G}v_g'f_g \in A_{\infty}\cap B'$. Then for any $b\in B$ by \ref{enum2:lift2},\ref{enum2:rokhlin2} and \ref{enum2:rokhlin3} it follows that
\begin{align*}
u^*ub&=\sum_{g,h\in G}f_g(v_g')^*v_h'f_hb\\
&=\sum_{g,h\in G}(v_g')^*v_h'bf_gf_h\\
&=\sum_{g\in G}(v_g')^*v_g'bf_g\\
&=b.
\end{align*}
Similarly $uu^*b=b$ for any $b\in B$. Moreover, \ref{enum2:lift3},\ref{enum2:lift1},\ref{enum2:rokhlin1} and \ref{enum2:rokhlin2} imply that for $b\in B$ and $g\in G$
\begin{align*}
u\alpha_g(u^*)b&=\sum_{h,k}v_h'f_h\alpha_g(f_k)\alpha_g(v_k')^*b\\
&=\sum_{h,k}v_h'\alpha_g(v_k')^*bf_h\alpha_g(f_k)\\ 
&=\sum_k v_{gk}'\alpha_g(v_k')^*bf_{gk}\\
&=\sum_k v_{g}'bf_{gk}\\
&=v_g'b,
\end{align*}
Therefore, by passing to the quotient, $u$ defines a unitary in $F(B,A_\infty)$ such that $u\alpha_g(u^*)=v_g$ for all $g\in G$.
\end{proof}
The proof of the next lemma is based on the proof of \cite[Proposition 3.4.1]{JON80}.
\begin{lemma}\label{lemma: equivariantembedding}
Let $G$ be a finite group and $A$ be a separable C$^*$-algebra such that $A\cong A\otimes \M_{|G|^\infty}$. Let $(\alpha,u)$ be an anomalous action with the Rokhlin property of $G$ on $A$ . Then there exists a $G$-equivariant unital embedding
$$(\M_{|G|^\infty}, s_G)\rightarrow (F(A),\alpha).$$
\end{lemma}
\begin{proof}
To prove this we inductively construct unital equivariant $^*$-homomorphisms $\phi_n:(\B(l^2(G)),\Ad(\lambda_G))\rightarrow (F(A),\alpha)$ for $n\in \N$ with commuting images. Then the map defined by $a_1\otimes \dots \otimes a_n \otimes \dots \longmapsto \prod_{i\in \N}\phi_i(a_i)$ will induce a $s_G$ to $\alpha$ equivariant map into $F(A)$.
\par Suppose $\phi_1,\phi_2,\dots,\phi_n:(\B(l^2(G)),\Ad(\lambda_G))\rightarrow (F(A),\alpha)$ are equivariant maps 
with commuting images and let $\psi_i:B(l^2(G))\rightarrow A_\infty\cap A'$ be linear lifts of $\phi_i$ for $1\leq i\leq n$, then
\begin{enumerate}[(i)]
\item $\psi_i(m)\psi_j(m')a-\psi_j(m')\psi_i(m)a=0,\quad \forall a\in A, m,m'\in B(l^2(G)),\newline 1\leq i\neq j\leq n$,\label{proofenumcommute}
\item $\alpha_g(\psi_i(m))a-\psi_i(\lambda_G(g)(m))a=0,\quad \forall a\in A, m\in B(l^2(G)), 1\leq i \leq n, g\in G$,\label{proofenumequiv}
\item $\psi_i(m)^*a-\psi_i(m)^*a=0,\quad \forall a\in A, m\in B(l^2(G)),1\leq i\leq n$,
\item $\psi_i(1)a-a=0,\quad \forall \in A, 1\leq i\leq n$.\label{proofenumunital}
\end{enumerate}
Let
$$S=\{\psi_i(m)a:  m\in B(l^2(G)),\ a\in A,\ 1\leq i\leq n\}.$$ 
Then $S$ is separable, $S=S^*$ and $S$ is $(\alpha,u)$ invariant. We check that $u_{g,h}S\subset S$ for all $g,h\in G$, the remaining conditions follow similarly. For $a\in A$, $m\in M$ and $1\leq i \leq n$ letting $m'=\Ad(\lambda_G)_{h^{-1}g^{-1}}(m)$ one has that
\begin{align*}
 u_{g,h}\psi_i(m)a&=u_{g,h}\psi_i(\Ad(\lambda_{G})_{gh}(m'))a\\
 &\overset{\ref{proofenumequiv}}{=}u_{g,h}\alpha_{gh}(\psi_i(m'))a\\
 &=\alpha_g\alpha_h(\psi_i(m'))u_{g,h}a\\
 &\overset{\ref{proofenumequiv}}{=}\psi_i(\Ad(\lambda_G)_{gh}(m'))u_{g,h}a \in S.
\end{align*}
As $A\cong A\otimes M_{|G|^\infty}$ pick a unital embedding from $B(l^2(G))$ into $F(S,A_\infty)$. (As $A\otimes \M_{|G|^\infty}\cong A$ there exists a unital embedding of $B(l^2(G))$ into $F(A)$ by \cite[Theorem 2.2]{TOWI07}. Moreover by reindexing one can also choose a homomorphism as stated.) It follows from Lemma \ref{lem:embeddingintofixedpoints} that there exists a unital embedding $B(l^2(G))\rightarrow F(S,A_{\infty})^{\alpha}$. Let $(e'_{g,h})_{g,h\in G}$ in $F(S,A_{\infty})^\alpha$ be the images of $e_{g,h}$ under this unital embedding. The permutation unitary $v_g=\sum_{h\in G} e'_{gh,h}$ gives a unitary representation of $G$ on $F(S,A_{\infty})^\alpha$ and as $\alpha_g(v_h)=v_h$ it follows that $v_g$ is an $\alpha$-cocycle. Therefore, by Lemma \ref{lemma: stability} there exists a unitary $u\in F(S,A_{\infty})$ such that $u\alpha_g(u^*)=v_g$. Now $f_{g,h}=u^*e'_{g,h}u$ for $g,h\in G$ is a set of matrix units such that
\begin{align*}
\alpha_k(f_{g,h})&=\alpha_k(u^*)e'_{g,h}\alpha_k(u)\\
&=u^*v_ke'_{g,h}v_k^*u\\
&=u^*(\sum_{h',h''\in G}e'_{kh',h'}e'_{g,h}e'_{h'',kh''})u\\
&=u^*(e'_{kg,kh})u\\
&=f_{kg,kh}.
\end{align*}
Hence the $^*$-homomorphism
\begin{align*}
\phi_{n+1}:\B(l^2(G))&\rightarrow F(S,A_{\infty})\\
e_{g,h}&\mapsto f_{g,h}.
\end{align*}
defines an $\Ad(\lambda_G)$ to $\alpha$ equivariant $^*$-homomorphisms. Moreover, the image of $\phi_{n+1}$ commutes with $\phi_i$ for all $1\leq i\leq n$. Considering $\phi_{n+1}$ as a unital equivariant homomorphism into $A_{\infty}\cap A'/A_{\infty}\cap A^{\perp}$ the induction argument is complete.
\end{proof}
We have collected all the necessary ingredients to prove the model action absorption.
\begin{proposition}\label{prop:absorption}
Let $G$ be a finite group and $A$ a separable C$^*$-algebra such that $A \cong A\otimes \M_{|G|^\infty}$. Let $(\alpha,u)$ be a $(G,\omega)$ action on $A$ with the Rokhlin property. Then $(\alpha,u)$ and $(\alpha\otimes s_G,u\otimes 1_{\M_{|G|^\infty}})$ are cocycle conjugate through an isomorphism that is approximately unitarily equivalent to $\id_A\otimes 1_{\M_{|G|^\infty}}$.
\end{proposition}
\begin{proof}
    By Lemma \ref{lemma: equivariantembedding} there exists a $G$-equivariant unital embedding $(\M_{|G|^\infty},s_G)\rightarrow (F(A),\alpha)$. 
    Thus the result follows from Theorem \ref{thm:Gabor}.
\end{proof}
\section{Classification}\label{sec:classification}
We now discuss the abstract approach to bootstrapping the classification of group actions on a given class of C$^*$-algebras to a classification of anomalous actions. This method is a generalisation of that used by Connes in \cite[Section 6]{CO77}, a similar strategy was recently used in \cite{IZU23} to classify $G$-kernels of poly-$\mathbb{Z}$ groups on $\mathcal{O}_2$.
\par Before proceeding with the result, we set up notation. For a group $G$, we say ``$(\alpha,u)$ is an anomalous $G$-action on $A$" and ``$(A,\alpha,u)$ is an anomalous $G$-C$^*$-algebra" interchangebly. Let $\Lambda$ be a functor whose domain category is the category of C$^*$-algebras (denoted \textbf{C*alg}). We say $\Lambda$ is \emph{invariant under approximate unitary equivalence} if $\Lambda(\alpha)=\Lambda(\theta)$ whenever $\alpha\au\theta$. We also say that $\Lambda$ restricted to a subcategory $\mathcal{C}\subset\operatorname{\textbf{C*alg}}$ \emph{is full on isomorphisms} , if whenever $\Phi\in \Hom(\Lambda(A),\Lambda(B))$ is an isomorphism for $A,B\in \mathcal{C}$, then there exists an isomorphism $\varphi:A\rightarrow B$ in $\mathcal{C}$ with $\Lambda(\varphi)=\Phi$. The sort of functors with these properties are those used in the classification of C$^*$-algebras. For example, the functor consisting of pointed $K_0$ and $K_1$ is invariant under approximate unitary equivalence, it is also full on isomorphisms when restricted to the category of unital Kirchberg algebras satisfying the UCT (see \cite{PHI00}). Similarly, the functors $KT_u$ and $\underline{K}T_u$ of \cite{Class1} are invariant under approximate unitary equivalence and is full on isomorphisms when restricted to classifiable C$^*$-algebras.
\par If $\Lambda$ is invariant under unitary equivalence, an anomalous action $(A,\alpha,u)$ induces a $G$-action on $\Lambda(A)$ through the automorphisms $\Lambda(\alpha_g)$. If $(A,\alpha,u)$ and $(B,\beta,v)$ are anomalous actions, we say the induced actions $\Lambda(\alpha_g)$ and $\Lambda(\beta_g)$ are \emph{conjugate} if there exists an isomorphism $\Phi:\Lambda(A)\rightarrow \Lambda(B)$ with $\Phi \Lambda(\alpha_g)\Phi^{-1}=\Lambda(\beta_g)$ for all $g\in G$. We denote this by $\Lambda(\alpha)\sim \Lambda(\beta)$.
\par Let $(A,\alpha,u)$ and $(A,\beta,v)$ be two anomalous $G$-C$^*$-algebras. We write $(\alpha,u)\simeq_{\Lambda}(\beta,v)$ if $(\alpha,u)\simeq (\beta,v)$ through an automorphism $\theta$ with $\Lambda(\theta)=\id_{\Lambda(A)}$. This notion recovers $K$-trivial cocycle conjugacy of Definition \ref{def:equivalences} when $\Lambda$ is taken to be the functor consisting of $K_0\oplus K_1$. Finally, if $\Ro$ is a class of anomalous $G$-C$^*$-algebras, we will say $\Ro$ is \emph{closed under conjugacy}, if whenever $(A,\alpha,u)\in \Ro$ and $\varphi:A\rightarrow B$ is an isomorphism in \textbf{C*alg} then $(B,\varphi\alpha\varphi^{-1},\varphi(u))\in \Ro$.
\begin{lemma}\label{lem:abstract}
    Let $G$ be a group, $\D$ a strongly self-absorbing C$^*$-algebra and $\Ro$ a class of anomalous $G$-C$^*$-algebras that is closed under conjugation.
    Let $\Lambda$ be a functor with domain category the category of C$^*$-algebras such that $\Lambda$ is invariant under approximate unitary equivalence and is full on isomorphisms for C$^*$-algebras in $\Ro$. Suppose further that,
    \begin{enumerate}[label=(A\arabic*)]
        \item there exists a $G$-action $(\D,s_G,1)$ such that if $(A,\alpha,u)\in \Ro$, then $(A,\alpha,u)\simeq (A\otimes\D,\alpha\otimes s_G,u\otimes 1)$ through an automorphism that is approximately unitarily equivalent to $\id_A\otimes 1_{\D}$;\label{item:1}
        \item if there exists a $(G,\omega)$ action in $\Ro$ for some $\omega\in Z^3(G,\T)$, then there exist a $(G,\omega)$ and $(G,\overline{\omega})$ action $(\D,s_G^{\omega},u^\omega)$ and $(\D,s_G^{\overline{\omega}},u^{\overline{\omega}})$ respectively such that $(\D,s_G^{\overline{\omega}},u^{\overline{\omega}})\otimes (\D,s_G^{\omega},u^\omega)\simeq (\D,s_G,1)$ and for any $(G,\omega)$-action $(A,\alpha,u)\in \Ro$, $(A,\alpha,u)\otimes (\D,s_G^{\overline{\omega}},u^{\overline{\omega}})\in \Ro$;\label{item:2}
        \item for cocycle actions $(A,\alpha,u),(B,\beta,v)\in \Ro$, $\Lambda(\alpha)\sim \Lambda(\beta)$ if and only if $\alpha\simeq\beta$.\label{item:3}
    \end{enumerate}
Then, if $(A,\alpha,u)$ and $(B,\beta,v)$ in $\Ro$, $(A,\alpha,u)\simeq(B,\beta,v)$ if and only if $\Lambda(\alpha)\sim \Lambda(\beta)$ and $\omicron(\alpha,u)=\omicron(\beta,v)$.
\par With the same hypothesis but replacing \ref{item:3} with the condition that
\begin{enumerate}[(A3')]
    \item for cocycle actions $(A,\alpha,u)$ and $(A,\beta,v)$ in $\Ro$, $(A,\alpha,u)\simeq_{\Lambda}(A,\beta,v)$ if and only if $\Lambda(\alpha_g)=\Lambda(\beta_g)$ for all $g\in G$,\label{item:3'}
\end{enumerate}
then if $(A,\alpha,u)$ and $(A,\beta,v)$ in $\Ro$, $(A,\alpha,u)\simeq_{\Lambda}(A,\beta,v)$ if and only if $\omicron(\alpha,u)=\omicron(\beta,v)$ and $\Lambda(\alpha_g)=\Lambda(\beta_g)$ for every $g\in G$.
\end{lemma}
\begin{proof}
First we show that if \ref{item:1}-\ref{item:3} hold and $(A,\alpha,u)$, $(B,\beta,v)$ are anomalous actions in $\Ro$, then $(A,\alpha,u)\simeq (B,\beta,v)$ if and only if $\Lambda(\alpha)\sim \Lambda(\beta)$ and $\omicron(\alpha,u)=\omicron(\beta,v)$. If $(A,\alpha,u)\simeq (B,\beta,v)$, it is clear that $\omicron(\alpha,u)=\omicron(\beta,v)$ and also that $\Lambda(\alpha)\sim \Lambda(\beta)$ as $\Lambda$ is trivial when evaluated at inner automorphisms. We now turn to the converse. Suppose $\Lambda(\alpha)\sim \Lambda(\beta)$ and $\omicron(\alpha,u)=\omicron(\beta,v)$. First note that this implies that also $\Lambda(\alpha\otimes \id_{\D})\sim \Lambda(\beta\otimes \id_{\D})$. Indeed, by \ref{item:1} let $\phi_A:A\rightarrow A\otimes \D$ and $\phi_B:B\rightarrow B\otimes\D$ be isomorphisms which are approximately unitarily equivalent to the first factor embeddings and $\Phi:\Lambda(A)\rightarrow \Lambda(B)$ be an isomorphism such that $\Phi \Lambda(\alpha_g)\Phi^{-1}=\Lambda(\beta_g)$ for $g\in G$. Note $\Lambda(\alpha_g\otimes\id_{\D})\Lambda(\phi_A)=\Lambda(\alpha_g\otimes\id_{\D})\Lambda(\id_A\otimes 1_{\D})=\Lambda(\alpha_g\otimes 1_{\D})=\Lambda(\phi_A)\Lambda(\alpha_g)$ (and similarly replacing $A$ by $B$). Hence we compute that
\begin{align*}
    \Lambda(\alpha_g\otimes \id_{\D})\Lambda(\phi_A)\Phi \Lambda(\phi_B)^{-1}&=\Lambda(\phi_A)\Lambda(\alpha_g)\Phi \Lambda(\phi_B)^{-1}\\
    &=\Lambda(\phi_A)\Phi \Lambda(\beta_g) \Lambda(\phi_B)^{-1}\\
    &= \Lambda(\phi_A)\Phi \Lambda(\phi_B)^{-1} \Lambda(\beta_g\otimes \id_{\D})
\end{align*}
it follows that $\Lambda(\phi_B)\Phi \Lambda(\phi_A)^{-1}$ conjugates $\Lambda(\alpha_g\otimes \id_{\D})$ to $\Lambda(\beta_g\otimes \id_{\D})$ for all $g\in G$. Now, by hypothesis we have that
\begin{align}
(A,\alpha,u) \overset{\ref{item:1}}&{\simeq} (A\otimes \D,\alpha\otimes s_G,u\otimes 1_{\D})\nonumber\\
\overset{\ref{item:2}}&{\simeq} (A\otimes(\D\otimes\D),\alpha\otimes(s_G^{\overline{\omega}}\otimes s_G^{\omega}),u\otimes (u^{\overline{\omega}}\otimes u^{\omega}))\nonumber\\
&=((A\otimes\D)\otimes\D,(\alpha\otimes s_G^{\overline{\omega}})\otimes s_G^{\omega},(u\otimes u^{\overline{\omega}})\otimes u^{\omega})\nonumber\\
\overset{\ref{item:3},\ref{item:2}}&{\simeq}((B\otimes\D)\otimes\D,(\beta\otimes  s_G^{\overline{\omega}})\otimes s_G^{\omega},(v\otimes u^{\overline{\omega}})\otimes u^{\omega})\label{eqn:chain}\\
&=(B\otimes(\D\otimes\D),\beta\otimes(s_G^{\overline{\omega}}\otimes s_G^{\omega}),v\otimes (u^{\overline{\omega}}\otimes u^{\omega})\nonumber)\\
&\overset{\ref{item:2}}{\simeq} (B\otimes \D,\beta\otimes s_G,v\otimes 1_{\D})\nonumber\\
&\overset{\ref{item:1}}{\simeq} (B,\beta,v).\nonumber
\end{align}
Where in the third isomorphism we have used \ref{item:3} for the cocycle actions $(A\otimes\D,\alpha_g\otimes s_G^{\overline{\omega}},u\otimes u^{\overline{\omega}})$ and $(B\otimes\D,\beta_g\otimes s_G^{\overline{\omega}},v\otimes u^{\overline{\omega}})$. The reason we may apply \ref{item:3} in this setting is that $s_G^{\overline{\omega}}$ is approximately inner and hence our previous computation shows that $\Lambda(\alpha_g\otimes 
 s_G^{\overline{\omega}})=\Lambda(\alpha_g\otimes \id_{\D})\sim \Lambda(\beta_g\otimes \id_{\D})=\Lambda(\beta_g\otimes s_G^{\overline{\omega}})$ as required for the application of \ref{item:3}.
\par Now suppose that we replace condition \ref{item:3} with \ref{item:3'}. We will show that under the hypothesis of the lemma, \ref{item:3'} implies \ref{item:3}. Therefore, the cocycle conjugacies in (\ref{eqn:chain}) still hold. Then we compute the isomorphisms that induce the cocycle conjugacies in (\ref{eqn:chain}) and show that their composition is the identity after applying $\Lambda$. Let $(A,\alpha,u)$ and $(B,\beta,v)$ be cocycle actions in $\Ro$. Suppose $\Lambda(\alpha)\sim \Lambda(\beta)$. There exists an isomorphism $\Phi\in \Hom(\Lambda(A),\Lambda(B))$ such that $\Phi \Lambda(\beta_g)\Phi^{-1}=\Lambda(\alpha_g)$ for all $g\in G$. As $\Lambda$ is full on isomorphisms, there exists a $^*$-isomorphism $\varphi:B\rightarrow A$ with $\Lambda(\varphi)=\Phi$. Therefore $\Lambda(\varphi\beta_g\varphi^{-1})=\Lambda(\alpha_g)$ for all $g\in G$. By \ref{item:3'} one has that $(A,\alpha,u)\simeq_\Lambda(A,\varphi\beta\varphi^{-1},\varphi(v))\simeq (B,\beta,v)$.
\par Set $A=B$ in (\ref{eqn:chain}). Reading from top to bottom in (\ref{eqn:chain}), denote by $\varphi_1$, $\varphi_2$, $\varphi_3$, $\varphi_4$ and $\varphi_5$ the isomorphisms inducing each of the conjugacies. Note that $\varphi_5=\varphi_1^{-1}$ and $\varphi_4=\varphi_2^{-1}$. By \ref{item:1}, $\varphi_1\au \id_A\otimes 1_{\mathcal{D}}$. Moreover, $\varphi_2\au \id_A\otimes\id_{\D}\otimes 1_{\D}$ by (\cite[Corollary 1.12]{TOWI07}). Denote by $\varphi$ the isomorphism inducing the cocycle conjugacy from $(A\otimes \D,\alpha\otimes 
s_G^{\overline{\omega}},u\otimes u^{\overline{\omega}})$ to  $(A\otimes \D,\beta\otimes s_G^{\overline{\omega}},u\otimes u^{\overline{\omega}})$ which satisfies $\Lambda(\varphi)=\Lambda(\id_A\otimes\id_{\D})$. We may use the functoriality of $\Lambda$ and its invariance under approximate unitary equivalence to see that
\begin{align*}
\Lambda(\varphi_5\varphi_4\varphi_3\varphi_2\varphi_1)=&\Lambda(\id_A\otimes 1_{\D}\otimes 1_{\D})^{-1}\Lambda(\varphi\otimes\id_{\D})\Lambda(\id_A\otimes 1_{\D}\otimes 1_{\D})\\
&=\Lambda(\id_A\otimes 1_{\D}\otimes 1_{\D})^{-1}\Lambda(\id_A\otimes \id_{\D}\otimes \id_{\D})\Lambda(\id_A\otimes 1_{\D}\otimes 1_{\D})\\
&=\id_{\Lambda(A)}.\qedhere
\end{align*}
\end{proof}
We now prove our classification theorems.
\begin{theorem}\label{thm:Rokhclasskirchberg}
    Let $G$ be a finite group and $A$ be a unital Kirchberg algebra satisfying the UCT with $A\cong A\otimes \M_{|G|^\infty}$. If $(\alpha,u)$, $(\beta,v)$ are anomalous actions of $G$ on $A$ with the Rokhlin property then $(\alpha,u)\simeq_{K}(\beta,v)$ if and only if $\omicron(\alpha,u)=\omicron(\beta,v)$ and $K_i(\alpha_g)=K_i(\beta_g)$ for all $g\in G$ and $i=0,1$.
\end{theorem}
\begin{proof}
    We check that the hypothesis of Lemma \ref{lem:abstract} is satisfied. Let $\D=\M_{|G|^\infty}$, $\Lambda$ be the functor consisting of the pointed $K_0$ group direct sum the $K_1$ group and $\Ro$ the class of Rokhlin anomalous $G$-actions on unital Kirchberg algebras satisfying the UCT. That $\Lambda$ is full on isomorphisms follows from \cite{PHI00}. Condition \ref{item:1} follows from Proposition \ref{prop:absorption}. For any $\omega\in Z^3(G,\T)$, we have actions $(\D,s_G^{\omega},u^{\omega})$ as discussed in Section \ref{sec:UHFactions}. That $(\D,s_G^{\overline{\omega}},u^{\overline{\omega}})\otimes (\D,s_G^{\omega},u^\omega)\simeq (\D,s_G,1)$ follows from \cite[Theorem III.6]{HEJO83} combined with \cite[Lemma 3.12]{IZU04} as the actions $(\D,s_G^{\omega},u^{\omega})$ have the Rokhlin property (and hence property $\mathcal{R}_\infty$) by Proposition \ref{prop:rokhlin}. Therefore, \ref{item:2} is also satisfied. Finally \ref{item:3'} is satisfied by Izumi's classification result \cite[Theorem 4.2]{IZU04II} and that every cocycle action with the Rokhlin property is a unitary perturbation of a group action \cite[Lemma 3.12]{IZU04II}. 
\end{proof}
\begin{theorem}\label{thm:RokhlinclassTAF}
    Let $G$ be a finite group and $A$ be a unital, simple, nuclear TAF-algebra in the UCT class such that $A\cong A\otimes \M_{|G|^\infty}$ and $(\alpha,u),(\beta,v)$ are anomalous actions on $A$ with the Rokhlin property, then $(\alpha,u)\simeq_{K}(\beta,v)$ if and only if $\omicron(\alpha,u)=\omicron(\beta,v)$ and $K_i(\alpha_g)=K_i(\beta_g)$ for all $g\in G$ and $i=0,1$.
\end{theorem}
\begin{proof}
    We apply Lemma \ref{lem:abstract} with $\D=\M_{|G|^\infty}$, $\Ro$ the class of Rokhlin anomalous actions on $\M_{|G|^\infty}$-stable unital, simple, separable, nuclear TAF-algebras satisfying the UCT and $\Lambda$ the functor consisting of ordered $K_0$ and $K_1$. Firstly, $\Lambda$ is full on isomorphisms by \cite{LIN04}. \ref{item:1} holds by Proposition \ref{prop:absorption}. \ref{item:2} holds for the same reason as in the proof of Theorem \ref{thm:Rokhclasskirchberg}. Condition \ref{item:3'} follows from a combination of \cite[Theorem 4.3]{IZU04II} and \cite[Lemma 3.12]{IZU04}.
\end{proof}
Similarly, one may classify anomalous actions with the Rokhlin property on the Razak--Jacelon algebra $\mathcal{W}$.
\begin{theorem}\label{thm:classW}
Let $G$ be a finite group and $(\alpha,u)$, $(\beta,v)$ be anomalous $G$ actions with the Rokhlin property on $\W$. Then $(\alpha,u)\simeq (\beta,v)$ if and only if $o(\alpha,u)=o(\beta,v)$.
\end{theorem}
\begin{proof}
 We check the conditions of Lemma \ref{lem:abstract} with $\D=\M_{|G|^\infty}$, $\Ro$ the class of Rokhlin anomalous actions on $\W$ and $\Lambda$ the trivial functor. Firstly, \ref{item:1} holds by Proposition \ref{prop:absorption}. Moreover, \ref{item:2} holds as in the proof of Theorem \ref{thm:Rokhclasskirchberg}. Finally, \ref{item:3} follows from \cite[Corollary 3.7]{NA16} as every cocycle action of a finite group on $\W$ is cocycle conjugate to a group action (this follows as $\W\cong \W\otimes M_{|G|}$ and hence \cite[Remark 1.5]{GASZ22} applies).
\end{proof}
In light of \cite[Theorem B]{Class1}, it follows from \cite[Theorem 3.5]{IZU04} that all Rokhlin anomalous actions of $G$ on classifiable $\M_{|G|^\infty}$-stable C$^*$-algebras are classified up to cocycle conjugacy by their induced action on the total invariant $\underline{K}T_u$ (see \cite[Section 3]{Class1}) and their anomaly.
\begin{corollary}\label{cor:Generalclassrokh}
    Let $G$ be a finite group. Let $A$ be a unital, simple, separable, nuclear, $\M_{|G|^\infty}$-stable C$^*$-algebra satisfying the UCT and $(\alpha,u)$, $(\beta,v)$ be anomalous $G$-actions with the Rokhlin property on $A$. Then $(\alpha,u)\simeq (\beta,v)$ if and only if $\underline{K}T_u(\alpha)\sim \underline{K}T_u(\beta)$ and $\omicron(\alpha,u)=\omicron(\beta,v)$.
\end{corollary}
\begin{proof}
    We apply Lemma \ref{lem:abstract} with $\D=\M_{|G|^\infty}$, $\Ro$ the class of Rokhlin anomalous actions on $\M_{|G|^\infty}$-stable unital, simple, separable, nuclear C$^*$-algebras satisfying the UCT and $\Lambda=\underline{K}T_u$. Firstly, $\Lambda$ is full on isomorphisms by \cite[Theorem A]{Class1}. \ref{item:1} holds by Proposition \ref{prop:absorption}. \ref{item:2} holds as in the proof of Theorem \ref{thm:Rokhclasskirchberg}.
    It remains to show \ref{item:3}. By \cite[Lemma 3.12]{IZU04} it suffices to show that for any two Rokhlin $G$-actions $(A,\alpha)$ and $(B,\beta)$ such that $\underline{K}T_u(\alpha)\sim \underline{K}T_u(\beta)$  then $\alpha\simeq \beta$. This has been shown for simple, unital AH-algebras in \cite[Theorem 3.8]{GASA16}. With \cite[Theorem B]{Class1} in hand this also follows for arbitrary unital, simple, separable, nuclear, $\Z$-stable C$^*$-algebras satisfying the UCT. Indeed, as $\underline{K}T_u$ is full on isomorphisms, there exists an isomorphism $\theta:A\rightarrow B$ such that $\underline{K}T_u(\theta\alpha_g\theta^{-1})=\underline{K}T_u(\beta_g)$ for all $g\in G$. Therefore, it follows from \cite[Theorem B]{Class1} that $\theta\alpha_g\theta^{-1}\au\beta_g$. Now, it follows immediately from \cite[Theorem 3.5]{IZU04} that $\alpha\simeq\beta$.
\end{proof}
We illustrate another application of Lemma \ref{lem:abstract} to the classification of Rokhlin anomalous actions on a class of non-simple C$^*$-algebras. Precisely, we can classify Rokhlin anomalous actions on inductive limits of 1-dimensional NCCW-complexes with trivial $K_1$-groups as a consequence of the classification results of \cite[Section 3.3.1]{GASA16}.
\begin{theorem}\label{thm:1dimNCCW}
Let $G$ be a finite group and $A$ be a C$^*$-algebra that can be written as an inductive limit of 1-dimensional NCCW-complexes with trivial $K_1$ groups satisfying $A\cong A\otimes M_{|G|^\infty}$. If $(\alpha,u)$, $(\beta,v)$ are anomalous actions of $G$ on $A$ then $(\alpha,u)\simeq (\beta,v)$ through an automorphism that is approximately inner if and only if $o(\alpha,u)=o(\beta,v)$ and $\Cu^{\sim}(\alpha_g)=\Cu^{\sim}(\beta_g)$ for all $g\in G$.\footnote{See \cite[Section 2.2]{GASA16} for the definition of the functor $\Cu^{\sim}$.}
\end{theorem}
\begin{proof}
We apply Lemma \ref{lem:abstract} with $\D=M_{|G|^\infty}$, $\Ro$ the class of Rokhlin anomalous actions on $M_{|G|^\infty}$-stable C$^*$-algebras that can be written as an inductive limit of 1-dimensional NCCW-complexes with trivial $K_1$ groups and $\Lambda=\Cu^{\sim}$. Firstly, $\Lambda$ is invariant under approximate unitary equivalence. Moreover, it is full on isomorphisms by \cite[Theorem 1]{RO12} (see also \cite[Corollary 5.2.3]{RO12}). Conditions \ref{item:1} and \ref{item:2} hold as in the proof of Theorem \ref{thm:Rokhclasskirchberg}. Condition \ref{item:3'} holds as a consequence of \cite[Theorem 3.6]{GASA16} (note also that $M_{|G|}(A)\cong A$ so \cite[Remark 1.5]{GASZ22} applies). Now, it follows from Lemma \ref{lem:abstract} that any two Rokhlin anomalous actions $(\alpha,u)$, $(\beta,v)$ of $G$ on an inductive limit of 1-dimensional NCCW-complex satisfy $(\alpha,u)\simeq_{\Cu^{\sim}}(\beta,v)$. But any automorphism of an inductive limit of 1-dimensional NCCW complexes with trivial $K_1$ groups that is the identity under $\Cu^{\sim}$ is approximately inner by \cite[Theorem 1]{RO12}.
\end{proof}
\begin{remark}
Note that, by \cite[Theorem 5.2]{GASA16}, the UHF-stability assumption in Theorem \ref{thm:1dimNCCW} is immediate for the following subclasses:
\begin{enumerate}[(i)]
\item unital C$^*$-algebras that can be written as inductive limits of $1$-dimensional NCCW-complexes;
\item simple C$^*$-algebras with trivial $K_0$-groups that can be written as inductive limits of $1$-dimensional NCCW-complexes;
\item C$^*$-algebras that can be written as inductive limits of punctured-tree algebras.
\end{enumerate}
\end{remark}
We have shown a classification of anomalous actions on some classes of simple C$^*$-algebras. Such a classification also implies a classification of $G$-kernels, we illustrate it by using Theorem \ref{thm:Rokhclasskirchberg}, the same argument may also be used to rewrite the results of Theorem \ref{thm:classW}, Theorem \ref{thm:RokhlinclassTAF} and Corollary \ref{cor:Generalclassrokh}.
\begin{corollary}\label{cor:corGkerclassrokh}
    Let $A$ be a unital Kirchberg algebra satisfying the UCT with $A\cong A\otimes M_{|G|^\infty}$ and $\overline{\alpha}$, $\overline{\beta}$ be $G$-kernels with the Rokhlin property on $A$. Then $\overline{\alpha}$ and $\overline{\beta}$ are $K$ trivially conjugate if and only if $\ob(\overline{\alpha})=\ob(\overline{\beta})$ and $K_i(\overline{\alpha}_g)= K_i(\overline{\beta}_g)$ for all $g\in G$ and $i=0,1$.
\end{corollary}
\begin{proof}
    The forward direction is clear. To show the reverse direction, pick lifts $(\alpha,u)$ of $\overline{\alpha}$ and $(\beta,v)$ of $\overline{\beta}$ such that $\omicron(\alpha,u)=\omicron(\beta,v)$. As $(\alpha,u)$ and $(\beta,v)$ satisfy the hypothesis of Theorem \ref{thm:Rokhclasskirchberg}, it follows that $(\alpha,u)\simeq (\beta,v)$ and so $\overline{\alpha}$ and $\overline{\beta}$ are conjugate.
\end{proof}
\section{Applications}\label{sec:applications}
We start this section by giving an alternative construction, of a $(G,\omega)$ action on the UHF algebra $M_{|G|^\infty}$ which is visibly compatible with a Bratteli diagram of $M_{|G|^\infty}$. This action is an AF-action in the sense of \cite{ELSU96} and \cite[Definition 4.8]{QCCJRHP22} (see also the discussion in \cite[Section 6.1]{THESIS}). The existence of an AF $\omega$-anomalous action on $M_{|G|^\infty}$ follows from an adaptation of the Ocneanu compactness argument to the C$^*$-setting (\cite{OC88}). We build it explicitily below.
\begin{proposition}\label{prop:inductiveconstruction}
Let $G$ be a finite group and $\omega\in Z^3(G,\T)$, then there exists an AF-$(G,\omega)$ action with the Rokhlin property on $M_{|G|^\infty}$. We denote this action by $\theta_G^{\omega}$.
\end{proposition}
\begin{proof}
In this proof we will use the symbols $g,h,k,x,y,x_i,y_i,s_i$ for $i\in \N$ to denote elements of the group $G$. Let $A_n=C(G)\otimes \bigotimes_{i=1}^{n-1}\B(l^2(G))$ for $n\in \N$, where by convention $A_1=C(G)$. For $f\in C(G)$, let $M_f\in \B(l^2(G))$ be the multiplication operator by $f$. Consider the $^*$-homomorphisms $\varphi_n:A_n\rightarrow A_{n+1}$ defined by $\varphi_n(f\otimes T)=1\otimes M_f\otimes T$
for $f\in C(G)$ and $T\in \bigotimes_{i=1}^{n-1}\B(l^2(G))$.
\par The inductive system $(A_n,\varphi_n)$ has an inductive limit (we write the limit by $A$) which is known to be isomorphic to $\M_{|G|^\infty}$. Indeed, the Bratelli diagram of this AF-algebra is easily seen to be the complete bipartite graph on $|G|$-vertices, it is common knowledge that this coincides with the UHF-algebra of type $|G|^{\infty}$ (see \cite[Example III.2.4]{Dav96} for the case $|G|=2$) We construct a $(G,\omega)$ action on each finite dimensional algebra $A_n$ such that the actions commute with the inclusion maps $\varphi_n$. This will induce an AF $\omega$-anomalous $G$ action on $M_{|G|^\infty}$ by the universal property of the inductive limit (see \cite[Section 6.1]{THESIS}).
\par To be precise, we construct a family of maps $\theta_n:G\rightarrow \Aut(A_n)$ and $u_n:G\times G\rightarrow U(A_n) $ such that:
\begin{enumerate}
\item $\theta_n(g)\theta_n(h)=\Ad(u_n(g,h))\theta_n(gh)$, \label{eqn:unitaryworks}
\item $\omega_{g,h,k}=\theta_n(g)(u_n(h,k))u_n(g,hk)u_n(gh,k)^*u_n(g,h)^*$,\label{eqn:anomalyUHF}
\item $\varphi_{n}(u_n(g,h))=u_{n+1}(g,h)$,\label{eqn:trivial}
\item $\varphi_n \theta_n(g)=\theta_{n+1}(g)\varphi_n$,\label{eqn:throughdiag}
\end{enumerate}
for all $n\in\N$. To build this we will consider the group actions $\theta_n':G\rightarrow \Aut(A_n)$ defined by $\theta_n'(g)=\lambda_G(g)\otimes\bigotimes_{i=1}^{n-1} \Ad(\lambda_G)_g$ where $\lambda_G$ is the left regular representation of $G$. Note that $\varphi_n\theta_n'(g)=\theta_{n+1}'(g)\varphi_n$. To take into account the anomaly, we will tweak $\theta_n'$ by suitable diagonal operators $d_n\in \Aut(A_n)$ and ensuring that (\ref{eqn:unitaryworks}) and (\ref{eqn:anomalyUHF}) hold. To define $d_n$ we start by introducing some notation. Let $\delta_k\in C(G)$ be the point mass at $k$ i.e.
\begin{equation*}
    \delta_k(g)=\begin{cases}
    1\ \text{if}\ g=k,\\
    0\ \text{otherwise}.
    \end{cases}
\end{equation*}
Let $e_{g,h}\in \B(l^2(G))$ be defined by 
\begin{equation*}
    e_{g,h}(f)(k)=\begin{cases}
    f(g)\ \text{if}\ k=h,\\
    0\ \text{otherwise},
    \end{cases}
\end{equation*}
for $f\in l^2(G)$. We now let
\begin{equation*}
\theta_n(g)=d_n(g)\theta_n'(g)
\end{equation*}
with $d_n(g)$ defined inductively
\begin{align*}
d_1(g)=&\id_{A_1},\\
d_2(g)(\delta_k\otimes e_{x_1,y_1})=&\omega_{x_1^{-1},g,g^{-1}k}\overline{\omega_{y_1^{-1},g,g^{-1}k}}(\delta_k\otimes e_{x_1,y_1}),\\
\end{align*}
and 
\begin{align*}
d_n(g)(\delta_k&\otimes e_{x_1,y_1}\otimes\dots\otimes e_{x_{n-1},y_{n-1}})\\
&=\omega_{x_{n-1}^{-1},g,g^{-1}x_{n-2}}\overline{\omega_{x_{n-3}^{-1},g,g^{-1}x_{n-2}}}\overline{\omega_{y_{n-1}^{-1},g,g^{-1}y_{n-2}}}\omega_{y_{n-3}^{-1},g,g^{-1}y_{n-2}}\\
& (d_{n-2}(g)(\delta_k\otimes e_{x_1,y_1}\dots\otimes e_{x_{n-3},y_{n-3}})\otimes e_{x_{n-2},y_{n-2}}\otimes e_{x_{n-1},y_{n-1}})
\end{align*}
for all $n>2$ with the convention that $x_{0}=y_{0}=k$. As we have defined $d_n(g)$ on a spanning set of $A_n$, $d_n(g)$ extend to linear maps from $A_n$ to itself. In fact each $d_n(g)$ is an endomorphism of $A_n$. First, it is clear that they preserve the $^*$-operation. To show the multiplicativity, it is sufficient to check on a spanning set. We show this by induction. For the case $n=2$ it is only non-trivial to check that
$$d_2(g)(\delta_k\otimes e_{x_1,y_1})d_2(g)(\delta_k\otimes e_{y_1,y_2})=d_2(g)(\delta_k\otimes e_{x_1,y_2}).$$
The left hand side is given by
\begin{align*}
&d_2(g)(\delta_k\otimes e_{x_1,y_1})d_2(g)(\delta_k\otimes e_{y_1,y_2})\\
&=\omega_{x_1^{-1},g,g^{-1}k}\overline{\omega_{y_1^{-1},g,g^{-1}k}}\omega_{y_1^{-1},g,g^{-1}k}\overline{\omega_{y_2^{-1},g,g^{-1}k}}(\delta_k\otimes e_{x_1,y_2})\\
&=\omega_{x_1^{-1},g,g^{-1}k}\overline{\omega_{y_2^{-1},g,g^{-1}k}}(\delta_k\otimes e_{x_1,y_2})
\end{align*}
which coincides with the right hand side. To show that $d_n(g)$ is multiplicative for $n>2$ it suffices to show that 
\begin{align*}
d_n(g)(\delta_k\otimes &e_{x_1,y_1}\otimes \dots \otimes e_{x_{n-1},y_{n-1}})d_n(g)(\delta_k\otimes e_{y_1,s_1}\otimes....\otimes e_{y_{n-1},s_{n-1}})\\
&=d_n(g)(\delta_k\otimes e_{x_1,s_1}\otimes....\otimes e_{x_{n-1},s_{n-1}}).
\end{align*}
This follows immediately from the induction hypothesis and a direct computation of the left hand side (as in the case for $n=2$). Notice that each $d_n(g)$ fixes elements of the form $\delta_k\otimes e_{x_1,x_1}\otimes e_{x_2,x_2}\dots \otimes e_{x_{n-1},y_{n-1}}$.
\par To construct a $(G,\omega)$ action on the first stage $A_1$, we let $u_1(g,h)(k)=\omega_{k^{-1},g,h}$. That $(\theta_1,u_1)$ defines a $(G,\omega)$ action on $C(G)$ is a straightforward computation (this is computed in \cite[Section 4]{T-DUALITY}). We proceed to  extend this action on $A_1$ to all of $M_{|G|^\infty}$ through the inductive limit. Let $u_n(g,h)=\varphi_{1,n}(u_1(g,h))$ and $\theta_n(g)=d_n(g)\theta_n'(g)$. For the remaining part of the proof we check that $(\theta_n,u_n)$ satisfy (\ref{eqn:unitaryworks})-(\ref{eqn:throughdiag}) for all $n\in \N$. We will repeatedly use the $3$-cocycle formula during the calculations, instead of commenting on this every time, we will instead colour code the parts of our equations to which we apply the $3$-cocycle formula.
\par We start by showing (\ref{eqn:unitaryworks}). Firstly,
\begin{align*}
\theta_n(g)\theta_n(h)&=d_n(g)\theta_n'(g)d_n(h)\theta_n'(h)\\
&=d_n(g)\theta_n'(g)d_n(h)\theta_n'(g)^{-1}\theta_n'(gh)\\
&=d_n(g)[g\cdot d_n(h)]\theta_n'(gh)
\end{align*}
denoting $g\cdot d_n(h)=\theta_n'(g)d_n(h)\theta_n'(g)^{-1}$. It is clear that (\ref{eqn:unitaryworks}) holds for all $n\in \N$ if and only if $d_n(g)g\cdot d_n(h)d_n(gh)^{-1}=\Ad(u_n(g,h))$ 
 on $A_n$ for all $n\in \N$. This holds trivially for $n=1$. For $n=2$ it follows from the $3$-cocycle formula that
\begin{align*}
&d_2(g)g\cdot d_2(h)d_2(gh)^{-1}(\delta_{ghk}\otimes e_{x_1,y_1})\\
=&d_2(g)g\cdot d_2(h)(\delta_{ghk}\otimes e_{x_1,y_1})
\overline{\omega_{x_1^{-1},gh,k}}\omega_{y_1^{-1},gh,k}\\
=&d_n(g)(\delta_{ghk}\otimes e_{x_1,y_1})\overline{\omega_{x_1^{-1},gh,k}}\omega_{y_1^{-1},gh,k}\omega_{x_1^{-1}g,h,k}\overline{\omega_{y_1^{-1}g,h,k}}\\
=&(\delta_{ghk}\otimes e_{x_1,y_1})\textcolor{red}{\overline{\omega_{x_1^{-1},gh,k}}\omega_{x_1^{-1}g,h,k}\omega_{x_1^{-1},g,hk}}\textcolor{blue}{\overline{\omega_{y_1^{-1},g,hk}}\omega_{y_1^{-1},gh,k}\overline{\omega_{y_1^{-1}g,h,k}}}\\
=&(\delta_{ghk}\otimes e_{x_1,y_1})\omega_{g,h,k}\omega_{x_1^{-1},g,h}\overline{\omega_{g,h,k}}\overline{\omega_{y_1^{-1},g,h}}\\
=&\Ad(\varphi_1(u_1(g,h))(\delta_{ghk}\otimes e_{x_1,y_1}).
\end{align*}
We now proceed with an inductive argument for arbitrary $n$. We assume that (\ref{eqn:unitaryworks}) holds for $n-2$, preforming a similar computation to the case $n=2$;
\begin{align*}
&d_n(g)g\cdot d_n(h)d_n(gh)^{-1}(\delta_{k}\otimes e_{x_1,y_1}\dots\otimes e_{ghx_{n-2},ghy_{n-2}}\otimes e_{x_{n-1},y_{n-1}})\\
=&\big(\Ad(u_{n-2}(g,h))((\delta_{k}\otimes e_{x_1,y_1}\dots\otimes e_{x_{n-3},y_{n-3}})\otimes e_{ghx_{n-2},ghy_{n-2}}\otimes e_{x_{n-1},y_{n-1}})\\
&\textcolor{red}{\overline{\omega_{x_{n-1}^{-1},gh,x_{n-2}}}\omega_{x_{n-1}^{-1}g,h,x_{n-2}}\omega_{x_{n-1}^{-1},g,hx_{n-2}}}\textcolor{blue}{\omega_{x_{n-3}^{-1},gh,x_{n-2}}\overline{\omega_{x_{n-3}^{-1}g,h,x_{n-2}}}\overline{\omega_{x_{n-3}^{-1},g,hx_{n-2}}}}\\
&\textcolor{red}{\overline{\omega_{y_{n-1}^{-1},gh,y_{n-2}}}\omega_{y_{n-1}^{-1}g,h,y_{n-2}}\omega_{y_{n-1}^{-1},g,hy_{n-2}}}\textcolor{blue}{\omega_{y_{n-3}^{-1},gh,y_{n-2}}\overline{\omega_{y_{n-3}^{-1}g,h,y_{n-2}}}\overline{\omega_{y_{n-3}^{-1},g,hy_{n-2}}}}\\
=&\big(\Ad(u_{n-2}(g,h))((\delta_{k}\otimes e_{x_1,y_1}\dots\otimes e_{x_{n-3},y_{n-3}})\otimes e_{ghx_{n-2},ghy_{n-2}}\otimes e_{x_{n-1},y_{n-1}})\\
&\omega_{g,h,x_{n-2}}\omega_{x_{n-1}^{-1},g,h}\overline{\omega_{g,h,x_{n-2}}\omega_{x_{n-3}^{-1},g,h}}\overline{\omega_{g,h,y_{n-2}}}\overline{\omega_{y_{n-1}^{-1},g,h}}\omega_{g,h,y_{n-2}}\overline{\omega_{y_{n-3}^{-1},g,h}}\\
=&\big(\Ad(u_{n-2}(g,h))(\delta_{k}\otimes e_{x_1,y_1}\dots\otimes e_{x_{n-3},y_{n-3}})\otimes e_{ghx_{n-2},ghy_{n-2}}\otimes e_{x_{n-1},y_{n-1}}\big)\\
&\omega_{x_{n-1}^{-1},g,h}\overline{\omega_{x_{n-3}^{-1},g,h}}\overline{\omega_{y_{n-1}^{-1},g,h}}\omega_{y_{n-3}^{-1},g,h}\\
=&(\delta_{k}\otimes e_{x_1,y_1}\dots\otimes e_{x_{n-1},y_{n-1}})\omega_{x_{n-1}^{-1},g,h}\overline{\omega_{y_{n-1}^{-1},g,h}}\\
=&\Ad(u_n(g,h))(\delta_{k}\otimes e_{x_1,y_1}\dots\otimes e_{ghx_{n-2},ghy_{n-2}}\otimes e_{x_{n-1},y_{n-1}}).
\end{align*}
\par For (\ref{eqn:throughdiag}) it suffices to show that $\varphi_nd_n(g)=d_{n+1}(g)\varphi_n$. For $n=1$
\begin{align*}
d_2(g)\varphi_1(\delta_k)&=\sum_{r\in G}d_2(g)(\delta_r\otimes e_{k,k})\\
&=(1\otimes e_{k,k})\\
&=\varphi_1d_1(g)(\delta_k)
\end{align*}
as $d_1$ is the identity map. The case $n=2$ follows too
\begin{align*}
d_3(g)\varphi_2(\delta_k\otimes e_{x,y})&=\sum_{r\in G}d_3(g)(\delta_r \otimes e_{k,k}\otimes e_{x,y})\\
&=(1\otimes e_{k,k}\otimes e_{x,y})\omega_{x^{-1},g,g^{-1}k}\overline{\omega_{y^{-1},g,g^{-1}k}}\\
&=\varphi_2d_2(g)(\delta_k\otimes e_{x,y}).
\end{align*}
Assuming that the case $n-2$ holds, we now argue by induction,
\begin{align*}
 &d_{n+1}(g)\varphi_n(\delta_k\otimes e_{x_1,y_1}\dots\otimes e_{x_{n-1},y_{n-1}})\\
 =&d_{n+1}(g)(\varphi_{n-3}(\delta_k\otimes e_{x_1,y_1}\dots\otimes e_{x_{n-3},y_{n-3}})\otimes e_{x_{n-2,y_{n-2}}}\otimes e_{x_{n-1},y_{n-1}})\\
 =&(d_{n-2}(g)\varphi_{n-3}(\delta_k\otimes e_{x_1,y_1}\dots\otimes e_{x_{n-3},y_{n-3}})\otimes e_{x_{n-2,y_{n-2}}}\otimes e_{x_{n-1},y_{n-1}})\\
 &\omega_{x_{n-1}^{-1},g,g^{-1}x_{n-2}}\overline{\omega_{x_{n-3}^{-1},g,g^{-1}x_{n-2}}}\overline{\omega_{y_{n-1}^{-1},g,g^{-1}y_{n-2}}}\omega_{y_{n-3}^{-1},g,g^{-1}y_{n-2}}\\
 =&(\varphi_{n-2}d_{n-2}(g)(\delta_k\otimes e_{x_1,y_1}\dots\otimes e_{x_{n-3},y_{n-3}})\otimes e_{x_{n-2,y_{n-2}}}\otimes e_{x_{n-1},y_{n-1}})\\
&\omega_{x_{n-1}^{-1},g,g^{-1}x_{n-2}}\overline{\omega_{x_{n-3}^{-1},g,g^{-1}x_{n-2}}}\overline{\omega_{y_{n-1}^{-1},g,g^{-1}y_{n-2}}}\omega_{y_{n-3}^{-1},g,g^{-1}y_{n-2}}\\
=&\varphi_n d_{n}(g)(\delta_k\otimes e_{x_1,y_1}\dots\otimes e_{x_{n-1},y_{n-1}})
\end{align*}
Condition (\ref{eqn:trivial}) is immediate. It remains to show that (\ref{eqn:anomalyUHF}) holds for arbitrary $n$. This follows from (\ref{eqn:anomalyUHF}) for the case $n=1$ and from (\ref{eqn:throughdiag}). For $n\in \N$
\begin{align*}
&\theta_n(g)(u_n(h,k))u_n(g,hk)u_n(gh,k)^*u_n(g,h)^*\\
&=\theta_n(g)(\varphi_{1,n}(u_1(h,k)))\varphi_{1,n}(u_1(g,hk))\varphi_{1,n}(u_1(gh,k)^*)\varphi_{1,n}(u_1(g,h)^*)\\
&=\varphi_{1,n}(\theta_1(g)(u_1(h,k))u_1(gh,k)u_1(g,hk)^*u_1(g,h)^*)\\
&=\omega_{g,h,k}\varphi_{1,n}(1_{A_1})\\
&=\omega_{g,h,k}.
\end{align*}
To show that $\theta_G^\omega$ has the Rokhlin property we construct a family of Rokhlin projections. The projections $\delta_g\otimes \id_{\B(l^2(G))^{\otimes n-1}}\in Z(A_n)$ satisfy $\theta_n(g)(\delta_h\otimes \id_{\B(l^2(G))^{\otimes n-1}})=\delta_{gh}\otimes \id_{\B(l^2(G))^{\otimes n-1}}$ and also $\sum_{g\in G}\delta_g\otimes \id_{\B(l^2(G))^{\otimes n-1}}=\id_{A_n}$. Therefore, the projections $p_g\in A_{\infty}$ with $n$-th coordinate given by $\varphi_{n,\infty}(\delta_g\otimes \id_{\B(l^2(G))^{\otimes n-1}})$ for $g\in G$ satisfy the conditions of Definition \ref{def:Rokhlin}.
\end{proof}
\begin{remark}\label{rem:trivialanom}
In the case that $\omega=1$ the construction in Proposition \ref{prop:inductiveconstruction} greatly simplifies. Indeed, $d_n(g)$ is the identity automorphism and $u_n(g,h)$ is the unit for all $g,h\in G$ and $n\in \N$. Therefore, $\theta_G^{1}$ restricts to the group action $\theta_n=\lambda_G\bigotimes_{i=0}^{n-1}\Ad(\lambda_G)$ on each $A_n$ with $\lambda_G$ the left regular representation. This action coincides with the infinite tensor product action $s_G$ (see Section \ref{sec:absorption}). To see this, consider the inductive system $(B_n,\phi_n)$ with $B_{2n-1}=A_n$, $B_{2n}=\bigotimes_{i=0}^nB(l^2(G))$ and $\phi_{2n-1}(f\otimes T)=M_f\otimes T$, $\phi_{2n}(S)=1\otimes S$ for all $n\in N, f\otimes T\in A_n$ and $S\in B_{2n}$. The even terms of the inductive system $(B_{2n},\phi_{2n+1}\circ\phi_{2n})$ coincide with the inductive limit $(\bigotimes_{i=1}^n B(l^2(G)),M\mapsto\ \id_{B(l^2(G))}\otimes M)$. The odd terms $(B_{2n-1},\phi_{2n}\circ \phi_{2n-1})$ coincide with the inductive system $(A_n,\varphi_n)$ from the proof of Proposition \ref{prop:inductiveconstruction}. This allows to interpolate between $(\bigotimes_{i=1}^n B(l^2(G)),M\mapsto\ \id_{B(l^2(G))}\otimes M)$ and $(A_n,\varphi_n)$. It is immediate that $\theta_G$ and $s_G$ are conjugate. Moreover, it follows from Theorem \ref{thm:RokhlinclassTAF} that $\theta_G^{\omega}$ is cocycle conjugate to $s_G^{\omega}$ for any $\omega\in Z^3(G,\T)$.
\end{remark}
\par We end this paper by studying to what extent Rokhlin anomalous actions on AF-algebras are AF-actions and vice versa. To do this, we will require results of \cite{QCCJRHP22}. In \cite{QCCJRHP22}, the authors associate an invariant to any AF-action $F$, of a fusion category $\mathcal{C}$, on a AF-algebra $A$. Vaguely, this invariant consists of the $K_0$-groups of all $Q$-system extensions of $A$ by $F$ and all natural maps between these extensions. The authors also show that any two AF-actions on AF-algebras $A$ and $B$ are equivalent if and only if their invariants are isomorphic. As observed in \cite[Section 5.1]{QCCJRHP22}, if the acting category $\mathcal{C}$ is \emph{torsion-free} (see \cite[Definition 3.7]{ARDE19}), the invariant of \cite{QCCJRHP22} simplifies to just the module structure of $K_0(A)$ under the action of the fusion ring of $\mathcal{C}$. We apply this when the acting category is $\boldsymbol{\Hilb}(G,\omega)$ and the action is induced by an anomalous action $(\alpha,u)$ as explained in \cite[Proposition 5.6]{EVGI21}. The fusion ring of $\boldsymbol{\Hilb}(G,\omega)$ is $\z[G]$ and the module structure of $K_0(A)$ is given by $K_0(\alpha_g)$.
\begin{corollary}\label{cor:AF}
  Let $G$ be a finite group and $A$ a simple, unital AF-algebra such that $A\cong A\otimes M_{|G|^\infty}$. Let $(\alpha,u)$ be a $(G,\omega)$-action on $A$ such that $K_0(\alpha_g)=\id_A$ for all $g\in G$. If $(\alpha,u)$ has the Rokhlin property, then $(\alpha,u)$ is an AF-action. Moreover, if $[\omega|_H]\neq 0$ for any subgroup $H<G$ then the converse holds.
\end{corollary}
\begin{proof}
      If $(\alpha,u)$ is a $(G,\omega)$-action with the Rokhlin property on an AF-algebra $A$, then by Theorem \ref{thm:RokhlinclassTAF} it is cocycle conjugate to the AF $\omega$-anomalous $G$-action $\id_A\otimes\ \theta_G^{\omega}$ on $A$. Therefore $(\alpha,u)$ is AF as (by definition) being AF is preserved under cocycle conjugacy (see \cite[Remark 6.1.7]{THESIS}). 
      \par We now consider the converse statement. An AF $\omega$-anomalous $G$ action $(\alpha,u)$ induces an AF-action of the fusion category $\boldsymbol{\Hilb}(G,\omega)$ in the sense of \cite{QCCJRHP22} (to see how a $(G,\omega)$-action induces a $\boldsymbol{\Hilb}(G,\omega)$ action see \cite[Proposition 5.6]{EVGI21}, that this is AF is discussed \cite[Remark 6.1.7]{THESIS}). By the hypothesis on $\omega$, the fusion category $\boldsymbol{\Hilb}(G,\omega)$ is torsion free, so as $K_0(\alpha_g)=\id_A$ and $K_0(\id_A\otimes\ \theta_G^{\omega})=\id_A$, then \cite[Theorem A]{QCCJRHP22} yields that the AF $\omega$-anomalous $G$ actions induced by $(\alpha,u)$ and $\id_A\otimes\ \theta_G^\omega$ are cocycle conjugate. So $(\alpha,u)$ has the Rokhlin property.
\end{proof}
\begin{remark}
One may drop the hypothesis that $A\cong A\otimes M_{|G|^\infty}$ in Corollary \ref{cor:AF} if one instead assumes that the anomaly $\omega$ of $(\alpha,u)$ is such that $[\omega]$ has order $|G|$. Indeed, it follows from \cite[Corollary 5.4.4]{THESIS} that in this case $A$ will automatically absorb $M_{|G|^\infty}$. Also, note that under this assumption on $[\omega]$ it is automatic that $[\omega|_H]\neq 0$ for any subgroup $H<G$. 
\end{remark}
\par The behavior observed in the converse of Corollary \ref{cor:AF} is quite different from the behaviour of group actions. It was already observed in \cite{FACK81} that there exist AF-actions of $\z_2$ on $M_{2^\infty}$ which do not have the Rokhlin property.
\bibliographystyle{abbrv}
\bibliography{MyRefs}
\end{document}